\documentclass[10pt,reqno]{amsart}
\usepackage{amsmath,amsopn,amssymb,amsthm}
\usepackage[cp1251]{inputenc}
\usepackage[english]{babel}
\usepackage[pagebackref,breaklinks=true,colorlinks=true,linkcolor=blue,citecolor=blue,urlcolor=blue]{hyperref}
\usepackage{graphicx}%
\usepackage{color}

\newtheorem{theorem}{Theorem}[section]

\newtheorem{lemma}[theorem]{Lemma}

\newtheorem{proposition}[theorem]{Proposition}
\newtheorem{remark}[theorem]{Remark}

\renewenvironment{proof}[1][Proof]{\noindent\textbf{#1.} }{\ \rule{0.5em}{0.5em}}

\begin{document}
\title[Homogeneous conic Landsberg Problem in two dimension]{Special left invariant conic Finsler metrics and homogeneous conic Landsberg Problem in two dimension}
\author{Ming Xu}
\address[Ming Xu]{
School of Mathematical Sciences,
Capital Normal University,
Beijing 100048,
P. R. China}
\email{mgmgmgxu@163.com}

\date{}

\begin{abstract}
In this paper, we study left invariant conic Finsler metrics on the 2-dimensional non-Abelian Lie group $G$ with nowhere vanishing spray vector fields, and classify those satisfying the constant curvature condition, the Landsberg condition or the Berwald condition respectively. We prove that any left invariant conic Landsberg metric on $G$ must be Berwald. This discovery suggest a homogeneous conic Landsberg Problem, i.e., searching for homogeneous conic Landsberg metrics and exploring if they are Berwald. We solve the 2-dimensional case of this problem
and prove that all homogeneous conic Landsberg surfaces are Berwald.


\end{abstract}

\maketitle

\section{Introduction}
Classifying {\it space forms}, i.e., metrics with constant (flag) curvature, and searching for {\it unicorns} \cite{Ba2007}, i.e., Landsberg metrics which are not Berwald, are two hot projects in Finsler geometry, which have been extensively studied.

Here are some remarkable achievements. H. Akber-Zadeh proved that,  on a closed manifold, any Finsler metric with constant curvature must be Riemannian or locally Minkowski, if that constant is negative or zero respectively \cite{Ak1988} (this result is usually called the {\it Akber-Zadeh Theorem}). D. Bao, C. Robles and Z. Shen classified Randers metrics with constant flag curvature \cite{BRS2004}. R. Bryant found non-Riemannian projective flat Finsler metrics on spheres with constant flag curvature \cite{Br1997,Br2002}.   Projective flag Finsler metrics with constant flag curvature were further studied by B. Li, Z. Shen, L. Zhou and others \cite{Li2014,Sh2003,Zh2012}. The geodesic behavior of a Finsler 2-sphere with constant curvature was firstly noticed by A. Kotok \cite{Ka1973} and recently studied by R. Bryant and his coworkers \cite{BFIMZ}. The Landsberg Conjecture \cite{Ma1996}, which guesses that smooth Landsberg metrics are always Berwald, was proved by S. Deng, B. Li, X. Mo, Z. Shen, M. Xu and others, for $(\alpha,\beta)$ metrics, $(\alpha_1,\alpha_2)$ metrics, generalized $(\alpha,\beta)$ metrics, spherically symmetric metrics, etc, when the dimension is bigger than $2$ \cite{MZ2014,Sh2009,XD2021,ZWL2019}. Recently, V. Matveev and M. Xu proposed a unified  geometric proof for the Landsberg Conjecture in above cases \cite{XM2022}. The Landsberg Conjecture is not correct when the metric permits singularities. G. Asanov, B. Li, Z. Shen and others constructed many singular unicorns \cite{As2006,Sh2009,ZWL2019}.

In homogeneous Finsler geometry, these two projects are also studied. For example, M. Xu studied homogeneous Finsler spheres with constant curvature \cite{Xu2018,Xu2019}.
Deng and M. Xu proposed a homogeneous Landsberg Conjecture in \cite{XD2021}, which guesses that each homogeneous Landsberg metric must be Berwald. Recently, B. Najafi and A. Tayebi proved the homogeneous Landsberg Conjecture in the $2$-dimensional case \cite{TN2021} (see \cite{Xu2022} for an alternative proof).

Affected by the work \cite{TN2021}, we study in this paper the homogeneous Finsler geometry in the $2$-dimensional case. Notice that the $2$-dimensional homogeneous Finsler manifold is very rare. If it is neither Riemannian with positive constant curvature nor locally Minkowski, then it must be the $2$-dimensional
non-Abelian Lie group $G$ equipped with a left invariant metric (to avoid iteration, the Lie group $G$
in this section will always be assumed to be this one).

We concern left invariant conic Finsler metrics with certain geometric properties. A conic Finsler metric on a smooth manifold $M$ is only defined in a conic open subset of $TM\backslash0$, and it satisfies similar requirements as for a globally defined Finsler metric \cite{JV2018}. Most geometric notions in general and homogeneous Finsler geometry are micro-locally defined, i.e., their definitions at any $(x,y)\in T_xM\backslash0$ is only relevant to the metric or spray data in an arbitrarily small neighborhood of $(x,y)$ in $TM\backslash0$. These notions can be naturally generalized to the conic context. This generalization can bring us some interesting new examples. For example,  the constant curvature metrics on surfaces found by R. Bryant, L. Huang and X. Mo are conic Finsler metrics \cite{BHM2017,HM2015}. The singular unicorns previously mentions can also be viewed as conic Finsler metrics. Among the metrics in the main results below, i.e., left invariant conic Finsler metrics on $G$ satisfying the constant curvature property, the Landsberg property or
the Berwald property respective, most are non-Riemannian and they can not be defined globally, either by
either an analog of the Akbar-Zadeh Theorem or by Theorem 1.1 in \cite{TN2021}.

Now we summarize the main results in this paper.
Firstly, we classify left invariant conic Finsler metrics $F$ on $G$ with constant flag curvature. \bigskip

\noindent{\bf Theorem A}\quad
{\it For any $c\in\mathbb{R}$ and $y_0\in\mathfrak{g}\backslash[\mathfrak{g},\mathfrak{g}]$, there is a one-to-one correspondence between the following two sets:
\begin{enumerate}
\item the set of all left invariant conic Finsler metrics $F$ which are defined around $y_0$ with constant flag curvature $K\equiv c$ and nowhere vanishing spray vector fields;
\item the open subset $\mathcal{M}=\{(a_0,a_1,a_2,a_3)|a_0>0,a_1\neq0, 2a_0a_2-a_1^2+4a_0^2>0\}\subset\mathbb{R}^4$.
\end{enumerate}}\bigskip

The correspondence in Theorem A needs some explanation.
A left invariant conic Finsler metric $F$ on $G$ is one-to-one determined by its restriction $F=F(e,\cdot)$, which is a conic Minkowski norm on $\mathfrak{g}$. Using some polar coordinate $(r,t)$, $F$ can be presented as $F=r\sqrt{2f(t)}$, which is defined for $t$ close to $0$. Two left invariant conic Finsler metrics $F_1=r\sqrt{2f_1(t)}$ and $F_2=r\sqrt{f_2(t)}$ are viewed as the same if $f_1(t)$ and $f_2(t)$ coincide around $0$. The ordinary differential equation (ODE in short) for the unknown $f(t)$, characterizing the condition $K\equiv c$, is the fourth order equation given in (\ref{020}), and $\mathcal{M}$
is the set of all possible initial values $(f(0),\tfrac{{\rm d}}{{\rm d}t}f(0), \tfrac{{\rm d^2}}{{\rm d}t^2}f(0),\tfrac{{\rm d}^3}{{\rm d}t^3}f(0))$. Each initial value in $\mathcal{M}$ uniquely determines $f(t)$ and then $F$.

Theorem A reverifies the work \cite{HM2015} of L. Huang and X. Mo, which classifies these conic Finsler metrics for the first time. Compared with \cite{HM2015}, the classification here is more explicit and straightforward. Notice that different metrics in Theorem A may be viewed as the same through a diffeomorphic action. So the parameter numbers in Theorem A and in \cite{HM2015} are different.

Nextly, we classify the left invariant conic Landsberg metrics $F$ on $G$.
We prove those metrics can also be parametrized by the same $\mathcal{M}$. \bigskip

\noindent
{\bf Theorem B}\quad
{\it For any $y_0\in\mathfrak{g}\backslash[\mathfrak{g},\mathfrak{g}]$, there is a one-to-one correspondence between the following two sets:
\begin{enumerate}
\item the set of left invariant conic Landsberg metrics $F$ which are defined around $y_0$ and have nowhere vanishing spray vector fields;
\item The open subset $\mathcal{M}=\{(a_0,a_1,a_2,a_3)|a_0>0,a_1\neq0, 2a_0a_2-a_1^2+4a_0^2>0\}\subset\mathbb{R}^4$.
\end{enumerate}
}\bigskip

The correspondence in Theorem B is similar to that in Theorem A. The fourth order ODE for the unknown $f(t)$ in $F=r\sqrt{2f(t)}$, which charaterizes the Landsberg property, is given in (\ref{017}), and $\mathcal{M}$ is the space of all possible initial values.

Thirdly, we discuss when the left invariant conic Finsler metric $F$ on $G$ is Berwald.
\bigskip

\noindent
{\bf Theorem C}\quad
{\it Let $\{e_1,e_2\}$ be a basis of $\mathfrak{g}$ satisfying $[e_1,e_2]=e_2$, and $F$ a left invariant conic Finsler metric on $G$ which is defined around $e_1\in\mathfrak{g}$. If $F$ is Berwald has a nowhere vanishing spray vector field, then we can find $(a,b,c,d)\in\mathbb{R}^4$ satisfying $ac\neq0$ and $ad-bc>0$, such that
\begin{equation}\label{028}
(ay^1+by^2)\tfrac{\partial}{\partial y^1}F+(cy^1+dy^2)\tfrac{\partial}{\partial y^2}F=0,\quad
\forall y=y^1 e_1+y^2 e_2\in\mathfrak{g}\backslash\{0\}.
 \end{equation}
 Conversely, any $(a,b,c,d)\in\mathbb{R}^4$ satisfying $ac\neq0$ and $ad-bc>0$ defines by (\ref{028}) a left invariant conic Berwald metric $F$ around $e_1$.
}\bigskip

Theorem C is a summarization for Lemma \ref{lemma-8} and Proposition \ref{prop-1}. The metrics in Theorem C can be explicitly determined, because its indicatrix $F=1$
in $\mathfrak{g}$ is a segment of a spiral curve, an ellipsoid, the graph of $y^1=c(y^2)^\mu$ with constants $c>0,\mu>1$, or
the graph of $y^2=c_1y^1\ln y^1+c_2 y^1$ with the constant $c_1>0$. Here $(y^1,y^2)$ is some linear coordinate on $\mathfrak{g}$. See the case-by-case discussion in Section \ref{subsection-3-4} for details.

Finally, we compare the sets of initial value data for Theorem B and Theorem C. Then we find that they in fact correspond to the same sets of left invariant conic Finsler metrics (see Claim and its proof  in Section \ref{subsection-3-5}). This critical observation immediately results the following theorem.
\bigskip

\noindent
{\bf Theorem D}\quad
{\it Any left invariant conic Landsberg Finsler metric on $G$ is Berwald.
}\bigskip

Notice that a homogeneous conic Finsler metric may not be complete, the Akbar-Zahed's technique in \cite{TN2021} or the proof of Theorem 4.3 in \cite{Xu2022} does not work here. We have to detour through Theorem B and Theorem C, which explicitly classify all left invariant conic Landsberg and Berwald metrics on $G$.

Theorem D is a little surprising, because all known results in literature can only prove the Landsberg Conjecture or its analog from the positive side, when the metric is globally defined, or provide counter examples from the negative side, when the metric is conic or permits singularities. It suggests we could further study the following {\it homogeneous conic Landsberg Problem}.
\bigskip

\noindent{\bf Problem E}\quad
{\it Look for homogeneous conic Landsberg metrics and explore if they are Berwald.}
\bigskip

In this paper, we solve the 2-dimensional case of this problem by the following theorem, which  generalizes Theorem 1.1 in \cite{TN2021}.
\bigskip

\noindent
{\bf Theorem F}\quad
{\it All homogeneous conic Landsberg surfaces are Berwald.
}\bigskip

The proof of Theorem E uses a similar strategy as that of Theorem 4.3 in \cite{Xu2022}. We first prove that, up to local isometries, homogeneous conic Finsler surfaces are very rare (see Lemma \ref{lemma-12}), among which only left invariant conic Finsler metrics on a 2-dimensional non-Abelian Lie group need our concern, and then we apply Theorem D.
The first step here  involves a relatively long case-by-case discussion, which is much harder than its analog in \cite{Xu2022}. The reason is that the isotropy subgroup $H$ for a homogeneous conic Finsler manifold $(G/H,F)$ with an effective $G$-action may not be a compactly imbedded subgroup of $G$.
So homogeneous conic Finsler manifolds may look very different from those globally defined ones.

%
%
%
%

Above main results have a byproduct in the ODE theory. The equation (\ref{017}) characterizing the Landsberg property for the conic metric $F$ in Theorem B is a nonlinear ODE, and it is hard to solve in general. However, Theorem D tells us that we do not need to solve it, because all solutions (satisfying certain requirement for the initial values) have already been found. Theorem C implies that these solutions can be explicitly described.

We ends the introduction with some remark on the tools we use. In this paper, we have applied a systematical approach in homogeneous Finsler geometry. We first use the spray vector field and connection operator invented by L. Huang \cite{Hu2015-1,Hu2015-2,Hu2017} to interpret the geodesics, curvatures, or other geometric notions and properties \cite{Xu2022,Xu2022-2,Xu2022-3}, and then translate the interpretation to an algebraic condition or an equation on a Lie algebra. Through this process, geometric problems can be very much simplified. This paper can be viewed another trial for the effectiveness of the above machinery.

This paper is organized as the following. In Section 2, we summarize some background knowledge and necessary tools in general and homogeneous Finsler geometry. In Section 3, we study left invariant conic Finsler metrics on a 2-dimensional non-Abelian Lie group, and prove all the main results. In Section 4, we prove Theorem E


\section{Preliminaries}

\subsection{Minkowski norm and Finsler metric}

A {\it Minkowski norm} on a real vector space $\mathrm{V}$, $\dim\mathrm{V}=n<+\infty$,
is a continuous function $F:\mathrm{V}\rightarrow\mathbb{R}_{\geq0}$ satisfying:
\begin{enumerate}
\item the regularity, i.e., $F|_{\mathrm{V}\backslash\{0\}}$ is positive and smooth;
\item the positive 1-homogeneity, i.e., $F(\lambda y)=\lambda F(y)$, $\forall \lambda\geq0, y\in\mathrm{V}$;
\item the strong convexity, i.e., for any basis $\{e_1,\cdots,e_n\}$ of $\mathrm{V}$ and the corresponding linear coordinate $y=y^ie_i$, the Hessian matrix
    $(g_{ij})=(\tfrac12[F^2]_{y^iy^j})$ is positive definite, or equivalently, the fundamental tensor
    $$g_y(u,v)=\tfrac12\tfrac{\partial^2}{\partial s\partial t}|_{s=t=0}F^2(y+su+tv),\quad\forall u,v\in\mathrm{V},$$
is an inner product, for any $y\in\mathrm{V}\backslash\{0\}$.
\end{enumerate}

A {\it Finsler metric} on a smooth manifold $M$ is a continuous function $F:TM\rightarrow\mathbb{R}_{\geq0}$ such that $F|_{TM\backslash0}$ is a positive smooth function and $F|_{T_xM}$ for each $x\in M$ is a Minkowski norm.

The Finsler metric $F$ on $M$ induces a {\it geodesic spray} $\mathbf{G}_F$, which is a smooth tangent vector field on $TM\backslash0$. Using any standard local coordinate $x=(x^i)\in M$ and
$y=y^i\partial_{x^i}\in T_xM$, $\mathbf{G}_F$ can be presented as
\begin{equation}\label{006}
\mathbf{G}_F=y^i\partial_{x^i}-2\mathbf{G}^i_F\partial_{y^i},\mbox{ where } \mathbf{G}^i_F=\mathbf{G}^i_F(x,y) \mbox{ is positively 2-homogeneous for }y,
\end{equation}
and it can be determined by
\begin{equation*}
\mathbf{G}^i_F=\tfrac14g^{il}([F^2]_{x^ky^l}y^k-[F^2]_{x^l}).
\end{equation*}

Notice that (\ref{006}) alone defines a more general {\it spray structure}  on $M$. Since
many notions in Finsler geometry, geodesic, linearly and nonlinearly parallel translations, Riemannian curvature, S-curvature, etc, are defined through $\mathbf{G}_F$, they can be naturally generalized to spray geometry.

See \cite{BCS2000,Sh2001} for more details.
\subsection{Left invariant Finsler metric on a Lie group}

Let $G$ be a Lie group. We denote by $e$ the unit element of $G$, $\mathfrak{g}$ its Lie algebra, and $\{e_1,\cdots,e_n\}$ a basis of $\mathfrak{g}$. We call a Finsler metric $F$ on $G$ {\it left invariant} if $(L_g)^*F=F$ for all $ g\in G$. Here $L_g$ denotes left translation. This is a special case of homogeneous Finsler metric,
where the homogeneous manifold is $G/H=G/\{e\}$, and the reductive decomposition is the unique one $\mathfrak{g}=\mathfrak{h}+\mathfrak{m}=0+\mathfrak{g}$. So all theories in homogeneous Finsler geometry \cite{De2012}
can be applied and simplified here. For example, the left invariant $F$ on $G$ is one-to-one determined by its restriction $F=F(e,\cdot)$, which can be any arbitrary {\it Minkowski norm} on $\mathfrak{g}=T_eG$.

\subsection{Spray vector field and connection operator}

In \cite{Hu2015-1}, L. Huang proposed the notions of spary vector field and connection operator in  homogeneous Finsler geometry.

For a left invariant $F$ on $G$, its {\it spray vector field} is
the smooth map
$\eta:\mathfrak{g}\backslash\{0\}\rightarrow\mathfrak{g}$ determined by
\begin{equation}\label{008}
g_y(\eta(y),u)=g_y(y,[u,y]),\quad\forall y\in\mathfrak{g}\backslash\{0\},u\in\mathfrak{g}.
\end{equation}

\begin{remark}\label{remark-1}
Theorem 3.1 in \cite{XD2014} reveals the relation between $\eta$ and $\mathbf{G}_F$.
For a left invariant Finsler metric $F$ on $G$, the geodesic spray $\mathbf{G}_F$ can be globally presented as
$\mathbf{G}_F=\mathbf{G}_0-\mathbf{H}$, in which $\mathbf{G}_0$ is the canonical affine bi-invariant spray structure (see Theorem A in \cite{Xu2022}), $\mathbf{H}$ is tangent to each $T_gG$, and $\eta=\mathbf{H}|_{T_eG\backslash\{0\}=\mathfrak{g}\backslash\{0\}}$.
\end{remark}

The {\it connection operator} $N:\mathfrak{g}\backslash\{0\}\times\mathfrak{g}
\rightarrow\mathfrak{g}$ for $(G,F)$ can be determined by
\begin{eqnarray}
g_y(N(y,v),u)&=&g_y([u,v],y)+g_y([u,y],v)+g_y([v,y],u)\nonumber\\
& &-2\mathbf{C}_y(u,v,\eta(y)),\quad\forall y\in\mathfrak{g}\backslash\{0\}, u,v\in\mathfrak{g},\label{009}
\end{eqnarray}
or equivalently $N(y,v)=\tfrac12D\eta(y,v)-\tfrac12[y,v]$
with $D\eta(y,v)=\tfrac{{\rm d}}{{\rm d} t}|_{t=0}\eta(y+tv)$ (see (4) in \cite{Hu2017}).
Here $\mathbf{C}_y(u,v,w)=\tfrac12\tfrac{{\rm d}}{{\rm d} t}|_{t=0}g_{y+tw}(u,v)=\tfrac14 \tfrac{\partial^3}{\partial r\partial s\partial t}|_{r=s=t=0}F^2(y+ru+sv+tw)$ is the {\it Cartan tensor}.
\subsection{Geodesic and curvature for a left invariant metric}

Geometric notions and properties of a left invariant Finsler metric $F$ on a Lie group $G$ can be described using the
spray vector field $\eta$ and the connection operator $N$.

Theorem D in \cite{Xu2022} (together with Remark \ref{remark-1}, same below) described geodesics on $(G,F)$, which can be reformulated as the following.

\begin{theorem}\label{thm-1}
Let $\eta$ be the spray vector field for $(G,F)$. Then there exists a one-to-one correspondence between the following two sets:
\begin{enumerate}
\item the set of geodesics $c(t)$ on $(G,F)$ with $F(\dot{c}(t))\equiv\mathrm{const}>0$ and $c(0)=e$;
\item the set of integral curves $y(t)$ of $-\eta$ in $\mathfrak{g}$.
\end{enumerate}
The correspondence is given by $y(t)=(L_{c(t)^{-1}})_*(\dot{c}(t))$.
\end{theorem}

Theorem 1.1 in \cite{Xu2022-2} characterizes linearly parallel translations, which can be reformulated as the following (see Theorem C in \cite{Xu2022-3} for its generalization).

\begin{theorem}\label{thm-3}
Let $\eta$ and $N$ be the spray vector field and connection operator for $(G,F)$. Suppose $c(t)$ is a geodesic with $y(t)=(L_{c(t)^{-1}})_*(\dot{c}(t))$. Then the smooth vector field $W(t)$ along $c(t)$ is linearly parallel if and only if $w(t)=(L_{c(t)^{-1}})_*(W(t))$ satisfies
\begin{equation*}
\tfrac{{\rm d}}{{\rm d} t}w(t)+N(y(t),w(t))+[y(t),w(t)]=0
\end{equation*}
for each value of $t$.
\end{theorem}


For the Riemann curvature of $(G,F)$, we will need the following description (see Theorem 1.2 in \cite{Xu2022-2} or Theorem F in \cite{Xu2022-3}).

\begin{theorem}\label{thm-2}
Let $\eta$ and $N$ be the spray vector field and the connection operator for $(G,F)$, $c(t)$ a geodesic with $y(t)=(L_{c(t)^{-1}})_*(\dot{c}(t))$, and $W(t)=(L_{c(t)})_*(w(t))$ is linearly parallel along $c(t)$. Then $N(t)=N(y(t),w(t))$ and $R(t)=\mathbf{R}_{y(t)}w(t)=(L_{c(t)^{-1}})_*(\mathbf{R}_{\dot{c}(t)}W(t))$, which are viewed as smooth vector fields along $y(t)$, are given by
$$N(t)=-\mathrm{L}_{\eta}w(t)\quad\mbox{and}\quad R(t)=\mathrm{L}_\eta N(t),$$
in which $\mathrm{L}$ denotes the Lie derivative.
\end{theorem}

For the Berwald property of $(G,F)$, we will use the following criterion.
\begin{lemma}\label{lemma-7}
For a left invariant conic Finsler metric $F$ on a Lie group $G$, the following statements are equivalent:
\begin{enumerate}
\item $F$ is Berwald, i.e., its Berwald tensor $\mathbf{B}^i_{jkl}=\tfrac{\partial^3}{\partial y^jy^ky^l}
\mathbf{G}_F^i$ vanishes everywhere;
\item the spray vector field $\eta$ of $F$ is quadratic, i.e., with respect to any basis $\{e_1,\cdots,e_n\}$ of $\mathfrak{g}=\mathrm{Lie}(G)$, all coefficients $\eta^i(y)$ in $\eta(y)=\eta^i(y)e_i$ are quadratic real functions of $y$;
\item the connection operator $N$ of $F$ is a bilinear map.
\end{enumerate}
\end{lemma}

\begin{proof}We first prove the equivalence between (1) and (2).
The metric $F$ is Berwald if and only if $\mathbf{G}_F$ is affine.
Since we have $\mathbf{G}_F=\mathbf{G}_0-\mathbf{H}=\mathbf{G}_0-\mathbf{H}^i\partial_{u^i}$
and $\mathbf{G}_0$ is affine, $\mathbf{G}_F$ is affine if and only if $\mathbf{H}$ is quadratic for its $y$-entry. Since $\mathbf{H}$ is left invariant, it is quadratic for its $y$-entry if and only if
the spray vector field $\eta=\mathbf{H}|_{T_eG\backslash\{0\}=\mathfrak{g}\backslash\{0\}}$
is quadratic.

We then prove the equivalence between (2) and (3). The connection operator $N$ is bilinear when $\eta$ is quadratic, because $N(y,v)=\tfrac12D\eta(y,v)-\tfrac12[y,v]=\tfrac12\tfrac{{\rm d}}{{\rm d} t}|_{t=0}\eta(y+tv)-\tfrac12[y,v]$. Conversely, $\eta$ is quadratic when $N$ is bilinear, because $\eta(y)=N(y,y)$.
\end{proof}

For the Landsberg property, we will only use the following well known description (see (7.16) in \cite{Sh2001-2}).

\begin{lemma}\label{lemma-1}
The Landsberg curvature on a Finsler manifold $(M,F)$, $$\mathbf{L}_y(u,v,w),\quad \forall y\in T_xM\backslash\{0\}, u,v,w\in T_xM,$$
is a symmetric multiple linear real function for its $u$-,$v$- and $w$-entries, and it can be determined by
\begin{equation}
\mathbf{L}_{\dot{c}(t)}(W(t),W(t),W(t))=
\tfrac{{\rm d}}{{\rm d} t}\mathbf{C}_{\dot{c}(t)}(W(t),W(t),W(t)),\quad\forall t,
\end{equation}
in which $c(t)$ is any geodesic and $W(t)$ is any linearly parallel vector field along $c(t)$.
\end{lemma}

\subsection{Remark for the conic situation}
A {\it conic Finsler metric} can be similarly defined on a smooth manifold $M$, expect that it is
only defined on a {conic open subset} $\mathcal{A}$ of $TM\backslash\{0\}$ \cite{JV2018}. Here we call $\mathcal{A}$ {\it conic} if $\mathcal{A}_x=\mathcal{A}\cap T_xM$ for each $x\in M$ is a nonempty conic open subset of $T_xM\backslash\{0\}$. In each $T_xM$, $F(x,\cdot):\mathcal{A}_x\rightarrow \mathbb{R}_{>0}$ satisfies the regularity, positive 1-homogeneity and strong convexity requirements. For simplicity, we may call $F(x,\cdot)$ a {\it conic Minkowski norm}.  All notions and discussions in general and homogeneous Finsler geometries can be generalized to the conic situation, as long as the tangent vector (for example, the $y$-entry in the fundamental tensor $g_y$, the geodesic spray coefficient $G^i_F=G^i_F(x,y)$, and the Riemann curvature operator $\mathbf{R}_y$, the tangent vector field $\dot{c}(t)$ for a geodesic $c(t)$, etc) is contained in $\mathcal{A}$.

In particular, a left invariant conic Finsler metric $F$ on $G$ is one-to-one determined by the conic Minkowski norm $F=F(e,\cdot)$. If $F(e,\cdot)$ is defined on the conic open subset $\mathcal{A}_e\subset\mathfrak{g}\backslash\{0\}=T_eG\backslash\{0\}$, then the left invariant conic Finsler metric $F$ is defined on $\mathcal{A}=\bigcup_{g\in G}(L_g)_*(\mathcal{A}_e)$. Its spray vector field $\eta:\mathcal{A}_e\rightarrow\mathfrak{g}$
and connection operator $N=\tfrac12D\eta(y,v)-\tfrac12[y,v]:\mathcal{A}_e\times\mathfrak{g}\rightarrow\mathfrak{g}$
can be defined by the same equalities (\ref{008})  and (\ref{009}). After slight modifications, all results mentioned above are valid for a left invariant conic Finsler metric.

\section{Special left invariant conic Finsler metrics on a two dimensional non-Abelian Lie group}
\subsection{Some notations and calculations in polar coordinate}
\label{subsection-3-1}
Throughout this section, we always assume that
$G$ is a 2-dimensional  non-Abelian Lie group. We choose a basis $\{e_1,e_2\}$ for its Lie algebra $\mathfrak{g}$ and denote by $(y^1,y^2)$ and $(r,t)\in\mathbb{R}_{>0}\times\mathbb{R}/2\pi\mathbb{Z}$ the coordinates in $$y=y^1 e_1+y^2 e_2= r\cos t\ e_1+r\sin t\ e_2\in\mathfrak{g}\backslash\{0\}.$$
 Suppose that $F=r\sqrt{2f(t)}$ a conic Minkowski norm on $\mathfrak{g}$, which is defined for $t$ around $0$ (or equivalently, in a sufficiently small connected conic open neighborhood of $e_1$). We use the same $F$ to denote the left invariant conic Finsler metric satisfying $F(e,\cdot)=r\sqrt{2f(t)}$.

The criterion for the strong convexity of $F=r\sqrt{2f(t)}$ can be checked by calculating the fundamental tensor $g_y$ for the tangent frame $\{\partial_r,\partial_t-\tfrac{r}{2f(t)}\tfrac{{\rm d}}{{\rm d} t}f(t)\partial_r\}$ at each point. Similar calculation as for Lemma 3.5 in
\cite{XM2022} shows that
\begin{eqnarray}
& &g_y(\partial_r,\partial_r)=2f(t), \quad g_y(\partial_r,\partial_t-\tfrac{r}{2f(t)}\tfrac{{\rm d}}{{\rm d} t}f(t)\partial_r)=0, \quad\mbox{and}\nonumber\\
& &g_y(\partial_t-\tfrac{r}{2f(t)}\tfrac{{\rm d}}{{\rm d} t}f(t)\partial_r,\partial_t-
\tfrac{r}{2f(t)}\tfrac{{\rm d}}{{\rm d} t}f(t)\partial_r)\nonumber\\
& &\quad \quad =\tfrac{r^2}{2f(t)}\left(2f(t)\tfrac{{\rm d}^2}{{\rm d}t^2}f(t)-\left(\tfrac{{\rm d}}{{\rm d} t}f(t)\right)^2+4f(t)^2\right)
\label{010}
\end{eqnarray}
at each $y=r\cos t\ e_1+r\sin t\ e_2\in\mathfrak{g}\backslash\{0\}$.
So the strong convexity of $F$ is equivalent to the positiveness of $g_y(\partial_t-\tfrac{r}{2f(t)}\tfrac{{\rm d}}{{\rm d} t}f(t)\partial_r,\partial_t-\tfrac{r}{2f(t)}\tfrac{{\rm d}}{{\rm d} t}f(t)\partial_r)>0$ everywhere. To summarize, we have

\begin{lemma}\label{lemma-3}
$F=r\sqrt{2f(t)}$ is strongly convex if and only if
$$2f(t)\tfrac{{\rm d}^2}{{\rm d}t^2}f(t)-\left(\tfrac{{\rm d}}{{\rm d} t}f(t)\right)^2+4f(t)^2>0$$
is satisfied everywhere.
\end{lemma}

The indicatrix $F=1$ in $\mathfrak{g}$ can be parametrized as $y(t)=\tfrac{1}{\sqrt{2f(t)}}\cos t\ e_1+\tfrac{1}{\sqrt{2f(t)}}\sin t\ e_2$. Its tangent vector field $\tfrac{{\rm d}}{{\rm d}}y(t)$ coincides with the restriction of
$\partial_t-\tfrac{r}{2f(t)}\tfrac{{\rm d}}{{\rm d} t}f(t)\partial_r$ to $y(t)$. This observation provides
$$\tfrac{{\rm d}}{{\rm d} t}y(t)=\partial_t-\tfrac{1}{(2f(t))^{3/2}}\tfrac{{\rm d}}{{\rm d} t}f(t)\partial_r$$
along $y(t)$.


The Cartan tensor can be similarly calculated as for (2.12) in \cite{XM2022}. Notice that it is multiple linear and symmetric, and it vanishes if any entry is $\partial_r$. So it is completely determined by
\begin{equation}\label{011}
\mathbf{C}_{y(t)}(\tfrac{{\rm d}}{{\rm d} t}y(t),\tfrac{{\rm d}}{{\rm d} t}y(t),\tfrac{{\rm d}}{{\rm d} t}y(t))=\mathbf{C}_{y(t)}(\partial_t,\partial_t,\partial_t)=
\tfrac{1}{f(t)}\tfrac{{\rm d}}{{\rm d} t}f(t)+\tfrac1{4f(t)}\tfrac{{\rm d}^3}{{\rm d}t^3}f(t).
\end{equation}

Finally, we consider the the spray vector field $\eta$ of $F$.

\begin{lemma} \label{lemma-9}
Assuming $[{e}_1,{e}_2]=\epsilon_1 {e}_1+\epsilon_2 {e}_2$,
then the spray vector field $\eta$ of
$F={r}\sqrt{2{f}({t})}$ satisfies
$$\eta =-\tfrac{(4\epsilon_1{y}^1+4\epsilon_2{y}^2)+
2(-\epsilon_1{y}^2+\epsilon_2{y}^1)\tfrac{1}{{f}({t})}
\tfrac{{\rm d}}{{\rm d}{t}}{f}({t})}{
\tfrac{2}{{f}({t})}\tfrac{{\rm d}^2}{{\rm d}{t}^2}{f}({t})-\left(
\tfrac{1}{{f}({t})}
\tfrac{{\rm d}}{{\rm d}{t}}{f}({t})\right)^2
+4}\cdot\left((-{y}^2-\tfrac{{y}^1}{2{f}({t})}
\tfrac{{\rm d}}{{\rm d}{t}}{f}({t})){e}_1
+({y}^1-\tfrac{{y}^2}{2{f}({t})}
\tfrac{{\rm d}}{{\rm d}{t}}{f}({t})){e}_2\right)$$
at each possible $y=y^1e_1+y^2e_2\in\mathfrak{g}\backslash\{0\}$. In particular, when restricted to the indicatrix $y(t)$ of $F$, we have
$$\eta(y(t))=-\tfrac{\sqrt{2f(t)}\left(f(t)(2\epsilon_1\cos t+2\epsilon_2\sin t)+
\tfrac{{\rm d}}{{\rm d}t}f(t)
(-\epsilon_1\sin t+\epsilon_2\cos t)
\right)}{
{2}{{f}({t})}\tfrac{{\rm d}^2}{{\rm d}{t}^2}{f}({t})-\left(
\tfrac{{\rm d}}{{\rm d}{t}}{f}({t})\right)^2
+4f(t)^2}\cdot\tfrac{{\rm d}}{{\rm d}t}y(t).$$
\end{lemma}
%
%

\begin{proof}At any $y=y^1e_1+y^2e_2=r\cos t \ e_1+r\sin t \ e_2\in\mathfrak{g}\backslash\{0\}$,
we identify $e_1$ and $e_2$ with $\partial_{y^1}$ and $\partial_{y^2}$ respectively. So we have
\begin{eqnarray*}
& &e_1=\cos t\partial_r-\tfrac1r\sin t\partial_t=-\tfrac1r\sin t(\partial_t-\tfrac{r}{2f(t)}\tfrac{{\rm d}}{{\rm d}t}f(t)\partial_r)
+(\cos t-\tfrac{1}{2f(t)})\tfrac{{\rm d}}{{\rm d}t}f(t)\sin t)\partial_r,\\
& &e_2=\sin t\partial_r+\tfrac1r\cos t\partial_t=
\tfrac1r\cos t(\partial_t-\tfrac{r}{2f(t)}\tfrac{{\rm d}}{{\rm d}t}f(t)\partial_r)
+(\sin t+\tfrac{1}{2f(t)}\tfrac{{\rm d}}{{\rm d}t}f(t)\cos t)\partial_r.
\end{eqnarray*}

By (\ref{008}) and
(\ref{010}), we have
\begin{eqnarray}
g_y(\eta(y),e_1)&=&g_y(y,[e_1,y])=g_y(r\partial_r,\epsilon_1y^2 e_1+\epsilon_2 y^2e_2)\nonumber\\
&=&\left(\epsilon_1ry^2(\cos t-\tfrac{1}{2f(t)}\tfrac{{\rm d}}{{\rm d}t}f(t)\sin t)+
\epsilon_2 ry^2(\sin t+\tfrac{1}{2f(t)}\tfrac{{\rm d}}{{\rm d}t}f(t)\cos t)\right)\cdot g_y(\partial_r,\partial_r)\nonumber\\
&=&\left(2\epsilon_1 y^1 y^2 f(t)+2\epsilon_2 (y^2)^2\right)+\left(-\epsilon_1(y^2)^2+\epsilon_2
y^1y^2\right)\tfrac{{\rm d}}{{\rm d}t}f(t).\label{013}
\end{eqnarray}
On the other hand, for
$$\partial_t-\tfrac{r}{2f(t)}\tfrac{{\rm d}}{{\rm d}t}f(t)\partial_r
=(-y^2-\tfrac{y^1}{2f(t)}\tfrac{{\rm d}}{{\rm d}t}f(t))e_1+(y^1-\tfrac{y^2}{2f(t)}\tfrac{d}{dt}f(t))e_2,$$
we have
\begin{eqnarray}
g_y(\partial_t-\tfrac{r}{2f(t)}\tfrac{{\rm d}}{{\rm d} t}f(t)\partial_r, e_1)
&=& g_y(\partial_t-\tfrac{r}{2f(t)}\tfrac{{\rm d}}{{\rm d} t}f(t)\partial_r, -\tfrac{\sin t}{r}(\partial_t-\tfrac{r}{2f(t)}\tfrac{{\rm d}}{{\rm d} t}f(t)\partial_r))\nonumber \\
&=&-\tfrac{y^2}{2f(t)}\left(2f(t)\tfrac{{\rm d}^2}{{\rm d}t^2}f(t)-\left(\tfrac{{\rm d}}{{\rm d} t}f(t)\right)^2+4f(t)^2\right).
\label{014}
\end{eqnarray}
Since $g_y(y,\eta(y))=0$, $\eta(y)$ is a scalar multiple of
$\partial_t-\tfrac{r}{2f(t)}\tfrac{{\rm d}}{{\rm d} t}f(t)\partial_r$. Compare (\ref{013}) and (\ref{014}), we
see the ratio between $\eta$ and $\partial_t-\tfrac{r}{2f(t)}\tfrac{{\rm d}}{{\rm d}t}f(t)\partial_r$ is
$$-\tfrac{(4\epsilon_1{y}^1+4\epsilon_2{y}^2)+
2(-\epsilon_1{y}^2+\epsilon_2{y}^1)\tfrac{1}{{f}({t})}
\tfrac{{\rm d}}{{\rm d}{t}}{f}({t})}{
\tfrac{2}{{f}({t})}\tfrac{{\rm d}^2}{{\rm d}{t}^2}{f}({t})-\left(
\tfrac{1}{{f}({t})}
\tfrac{{\rm d}}{{\rm d}{t}}{f}({t})\right)^2
+4},$$
which proves the first statement of Lemma \ref{lemma-9}.

When restricted to $y(t)$, where $y^1=\tfrac{1}{\sqrt{2f(t)}}\cos t$ and $y^2=\tfrac{1}{\sqrt{2f(t)}}\sin t$, the second statement of Lemma \ref{lemma-9} follows immediately.
\end{proof}

For the convenience in later discussion, we will always assume $[e_1,e_2]=e_2$ and $\eta(y(0))\neq0$. Because $F$ is defined only for $t$ sufficiently close to $0$, the assumption $\eta(y(0))\neq0$ is equivalent to the non-vanishing of $\eta$ everywhere. The second statement
in Lemma \ref{lemma-9} can be simplified as
\begin{equation}\label{026}
\eta(y(t))=\tfrac{\sqrt{2f(t)}\left(-2 f(t)\sin t-
\tfrac{{\rm d}}{{\rm d}t}f(t)\cos t
\right)}{
 {2}{{f}({t})}\tfrac{{\rm d}^2}{{\rm d}{t}^2}{f}({t})-\left(
\tfrac{{\rm d}}{{\rm d}{t}}{f}({t})\right)^2
+4f(t)^2}\cdot\tfrac{{\rm d}}{{\rm d}t}y(t),
\end{equation}
from which we see $\tfrac{{\rm d}}{{\rm d}t}f(0)\neq0$.

We define a new smooth parameter
\begin{equation}\label{015}
s=\int\left(\tfrac{\sqrt{2f(t)}\left(2 f(t)\sin t+
\tfrac{{\rm d}}{{\rm d}t}f(t)\cos t
\right)}{
 {2}{{f}({t})}\tfrac{{\rm d}^2}{{\rm d}{t}^2}{f}({t})-\left(
\tfrac{{\rm d}}{{\rm d}{t}}{f}({t})\right)^2
+4f(t)^2}\right)^{-1}{\rm d}t
\end{equation}
 for the indicatrix $F=1$, and denote $t=t(s)$ the inverse function of $s=s(t)$.

\begin{lemma}\label{lemma-5}
The reparametrization $y(t(s))$ makes the indicatrix of $F$  an integral curve of $-\eta$.
\end{lemma}
\begin{proof}
By the Chain Rule and (\ref{026}), we have
$$\tfrac{{\rm d}}{{\rm d} s}y(t(s))=\tfrac{{\rm d}}{{\rm d} s}t(s)\cdot\tfrac{{\rm d}}{{\rm d} t}|_{t=t(s)}y(t)=-\eta(y(t(s))),$$
which proves Lemma \ref{lemma-5}.
\end{proof}
\subsection{Conic Finsler metrics with constant flag curvature}
In this subsection, we discuss when the left invariant conic Finsler metric $F=r\sqrt{2f(t)}$ in Section
 \ref{subsection-3-1} has constant flag curvature.

Denote
\begin{equation}\label{016}
u(t)=\tfrac{\tfrac{{\rm d}}{{\rm d} t}y(t)}{\left(g_{y(t)}(\tfrac{{\rm d}}{{\rm d} t}y(t),\tfrac{{\rm d}}{{\rm d} t}y(t))\right)^{1/2}}
=\tfrac{\partial_t-\tfrac{1}{(2f(t))^{3/2}}\tfrac{{\rm d}}{{\rm d} t}f(t)\partial_r}{\tfrac{1}{2f(t)}\left(2f(t)
\tfrac{{\rm d}^2}{{\rm d}t^2}f(t)-\left(\tfrac{{\rm d}}{{\rm d} t}f(t)\right)^2+4f(t)^2\right)^{1/2}
},
\end{equation}
which is a smooth tangent vector field along $y(t)$ with $g_{y(t)}(u(t),u(t))\equiv 1$.

\begin{lemma} \label{lemma-6}
$F=r\sqrt{2f(t)}$ has constant flag curvature $K\equiv c$ if and only if
\begin{equation}\label{018}
-\left(\tfrac{{\rm d}^2}{{\rm d}s^2}\lambda(s)\right)\cdot \lambda(s)^{-1}\equiv c,\forall s.
\end{equation}
Here the $s$-parameter is given in (\ref{015}) and $\lambda(s)$ is determined by
$u(t(s))=\lambda(s)\cdot\eta(y(t(s)))$.
\end{lemma}

\begin{proof} By the left invariancy,  $F=r\sqrt{2f(t)}$ has constant flag curvature $K\equiv c$ if and only if
$K(e, y(t(s)),y(t(s))\wedge u(t(s)))= g_{y(t(s))}(\mathbf{R}_{y(t(s))}u(t(s)), u(t(s)))\equiv c$.
By Theorem \ref{thm-2}, $\mathbf{R}_{y(t(s))} u(t(s))=-\mathrm{L}_\eta\mathrm{L}_{\eta} u(t(s))$.
Since $u(t(s))=\lambda(s)\cdot\eta(y(t(s)))$, we can use Lemma \ref{lemma-5} to get
\begin{eqnarray*}
\mathbf{R}_{y}= -\mathrm{L}_\eta\mathrm{L}_{\eta} u(t(s))
=\left(\tfrac{{\rm d}^2}{{\rm d}s^2}\lambda(s) \right)\cdot \eta(y(t(s))
=\left(\tfrac{{\rm d}^2}{{\rm d}s^2}\lambda(s) \right)\cdot \lambda(s)^{-1}\cdot u(t(s)),
\end{eqnarray*}
and then
$$K(e,y(t(s)),y(t(s))\wedge u(t(s)))=\left(\tfrac{{\rm d}^2}{{\rm d}s^2}\lambda(s) \right)\cdot \lambda(s)^{-1}.$$
This ends the proof of Lemma \ref{lemma-6}.
\end{proof}
\medskip

\begin{proof}[Proof of Theorem A]
For $y_0\in\mathfrak{g}\backslash[\mathfrak{g},\mathfrak{g}]$, $\mathrm{ad}(y_0)=[y_0,\cdot]:\mathfrak{g}\rightarrow\mathfrak{g}$ has a nonzero eigenvalue. We first discuss the case that $\mathrm{ad}(y_0)$ has a positive eigenvalue. Then we can choose a basis $\{e_1,e_2\}$ of $\mathfrak{g}$ such that $[e_1,e_2]= e_2$
and $y_0\in\mathbb{R}_{>0}e_1$.

Using Lemma (\ref{026}) and (\ref{016}), we can write (\ref{018}) as
\begin{eqnarray}
& &-\tfrac{{\rm d}^2}{{\rm d}s^2}
\left(
\tfrac{\left(2f(t(s))^2\tfrac{{\rm d}^2}{{\rm d}t^2}f(t(s))-
f(t(s))\left(\tfrac{{\rm d}}{{\rm d} t}f(t(s))\right)^2
+4f(t(s))^3\right)^{1/2}}{
(2f(t(s)))^{3/2}\left(-2f(t(s))\sin t(s)-\tfrac{{\rm d}}{{\rm d} t}f(t(s))\cos t(s)\right)}\right)
\cdot \nonumber\\
& &\left(
\tfrac{\left(2f(t(s))^2\tfrac{{\rm d}^2}{{\rm d}t^2}f(t(s))-
f(t(s))\left(\tfrac{{\rm d}}{{\rm d} t}f(t(s))\right)^2
+4f(t(s))^3\right)^{1/2}}{
(2f(t(s)))^{3/2}\left(-2f(t(s))\sin t(s)-\tfrac{{\rm d}}{{\rm d} t}f(t(s))\cos t(s)\right)}\right)^{-1}\equiv c,\quad\forall s.\label{019}
\end{eqnarray}
It can be further translated to the following fourth order ODE,
\begin{eqnarray}& &
-\tfrac{{\rm d}}{{\rm d} t}\left(\tfrac{{\rm d}}{{\rm d} t}
\left(
\tfrac{\left(2f(t)^2\tfrac{{\rm d}^2}{{\rm d}t^2}f(t)-
f(t)\left(\tfrac{{\rm d}}{{\rm d} t}f(t)\right)^2
+4f(t)^3\right)^{1/2}}{
(2f(t))^{3/2}\left(-2f(t)\sin t-\tfrac{{\rm d}}{{\rm d} t}f(t)\cos t\right)}\right)
\cdot
\left(\tfrac{\sqrt{2f(t)}\left(2f(t)\sin t+\tfrac{{\rm d}}{{\rm d} t}f(t)\cos t\right)}{2f(t)\tfrac{{\rm d}^2}{{\rm d}t^2}f(t)-\left(\tfrac{{\rm d}}{{\rm d} t}f'(t)\right)^2+4f(t)^2}\right)\right)\nonumber\\
& &\cdot
\left(\tfrac{\sqrt{2f(t)}\left(2f(t)\sin t+\tfrac{{\rm d}}{{\rm d} t}f(t)\cos t\right)}{2f(t)\tfrac{{\rm d}^2}{{\rm d}t^2}f(t)-\left(\tfrac{{\rm d}}{{\rm d} t}f'(t)\right)^2+4f(t)^2}\right)
\cdot
\left(
\tfrac{\left(2f(t)^2\tfrac{{\rm d}^2}{{\rm d}t^2}f(t)-
f(t)\left(\tfrac{{\rm d}}{{\rm d} t}f(t)\right)^2
+4f(t)^3\right)^{1/2}}{
(2f(t))^{3/2}\left(-2f(t)\sin t-\tfrac{{\rm d}}{{\rm d} t}f(t)\cos t\right)}\right)^{-1}
\nonumber\\
& &=c.\label{020}
\end{eqnarray}
This equation is seemingly complicated, but theoretically simple. The conditions $f(0)>0$ and $\tfrac{{\rm d}}{{\rm d}t}f(0)\neq0$ guarantees that (\ref{020}) is regular at $t=0$. When $(f(0),\tfrac{{\rm d}}{{\rm d} t}f(0), \tfrac{{\rm d}^2}{{\rm d}t^2}f(0),\tfrac{{\rm d}^3}{{\rm d}t^3}f(0))$ is chosen from $\mathcal{M}$, i.e.,
\begin{equation}\label{021}
f(0)>0, \ \tfrac{{\rm d}}{{\rm d} t}f(0)\neq0, \ 2f(0)\tfrac{{\rm d}^2}{{\rm d}t^2}f(0)-\left(\tfrac{{\rm d}}{{\rm d} t}f(0)\right)^2+4f(0)^2>0,
\end{equation}
a unique $f(t)$ can be solved from (\ref{020}),
which provides a left invariant conic $F=r\sqrt{2f(t)}$
with constant flag curvature $K\equiv c$. The regularity and positive 1-homogeneity of this conic $F$,
which is only defined for $t$ sufficiently close to $0$, is guaranteed by the smoothness of $f(t)$, the first requirement in (\ref{021}), and the polar coordinate representation $F=r\sqrt{2f(t)}$. The third requirement in (\ref{021}) guarantees that $F$ is strongly convex (see Lemma \ref{lemma-3}).
The second guarantees that the spray vector field $\eta$ is nowhere vanishing.

Above argument proves Theorem A when $\mathrm{ad}(y_0)$ has a positive eigenvalue.
Nextly, we consider the case that $\mathrm{ad}(y_0)$ has a negative eigenvalue, the proof is similar. The ODE characterizing $K\equiv c$ and the requirement for the initial value remain unchanged.
\end{proof}

\subsection{Conic Landsberg metrics}
\label{subsection-3-3}
In this subsection,  we discuss when the left invariant conic Finsler metric $F=r\sqrt{2f(t)}$ in Section
 \ref{subsection-3-1} is Landsberg, i.e., its Landsberg curvature vanishes everywhere.

\begin{lemma}\label{lemma-2} $F=r\sqrt{2f(t)}$ is Landsberg if and only if
\begin{equation}\label{012}
\tfrac{{\rm d}}{{\rm d} t}\mathbf{C}_{y(t)}(u(t),u(t),u(t))\equiv0.
\end{equation}
Here $u(t)$ is the smooth tangent vector field along $y(t)$ satisfying $g_{y(t)}(u(t),u(t))\equiv1$, which is given in (\ref{016}).
\end{lemma}
\begin{proof} We first assume $(G,F)$ has vanishing Landsberg curvature and prove (\ref{012}).
Apply the $s$-parameter in (\ref{015}), we see by Lemma \ref{lemma-5} that
$y(t(s))$ is an integral curve of $-\eta$. Then Theorem \ref{thm-1} provides an $F$-unit geodesic $c(s)$ on $(G,F)$, satisfying $y(t(s))=(L_{c(s)^{-1}})_*(\dot{c}(s))$. The vector field $U(s)=(L_{c(s)})_*(u(t(s)))$ is a linearly parallel vector field along $c(s)$.
By Lemma \ref{lemma-1}, we have $\tfrac{{\rm d}}{{\rm d} s}\mathbf{C}_{\dot{c}(s)}(U(s),U(s),U(s))\equiv0$. Since
$(L_{c(s)^{-1}})_*$ is a linear isometry between $F(c(s),\cdot)$ and $F(e,\cdot)$,
\begin{eqnarray*}& &
\mathbf{C}_{\dot{c}(s)}(U(s),U(s),U(s))\\ &=&
\mathbf{C}_{(L_{c(s)^{-1}})_*(\dot{c}(s))}((L_{c(s)^{-1}})_*(U(s)),
(L_{c(s)^{-1}})_*(U(s)), (L_{c(s)^{-1}})_*(U(s)), )\\
&=& \mathbf{C}_{y(t(s))}(u(t(s)),u(t(s)),u(t(s))).
\end{eqnarray*}
So we have
\begin{eqnarray*}
\tfrac{{\rm d}}{{\rm d} s}\mathbf{C}_{y(t(s))}(u(t(s)),u(t(s)),u(t(s)))= \tfrac{{\rm d}}{{\rm d} s}t(s)\cdot\tfrac{{\rm d}}{{\rm d} t}|_{t=t(s)}\mathbf{C}_{y(t)}(u(t),u(t),u(t))=0,\quad\forall s,
\end{eqnarray*}
i.e., $ \tfrac{{\rm d}}{{\rm d} t}\mathbf{C}_{y(t)}(u(t),u(t),u(t))\equiv0$. This proves one side of Lemma \ref{lemma-2}.

Reversing above argument, the other side of Lemma \ref{lemma-2} can be proved similarly. Notice that when we use Lemma \ref{lemma-1} to prove the Landsberg curvature vanishes, we only need to consider an $F$-unit speed geodesic $c(s)$. The corresponding $(L_{c(s)^{-1}})_*(\dot{c}(s))$ coincides with $y(t(s-s_0))$ for some suitable $s_0$. For any linearly parallel vector field $W(s)$ along $c(s)$, we can find $a,b\in\mathbb{R}$, such that $W(s)=a\dot{c}(s)+bU(s)$, in which
$U(s)=(L_{c(s)})_*(u(s-s_0))$.
Since $\mathbf{C}_{\dot{c}(s)}(\dot{c}(s),\cdot,\cdot)\equiv0$, we only need to check the requirement in Lemma \ref{lemma-1} for $U(s)$, which is guaranteed by (\ref{012}) through the action of the linear isometry $(L_{c(s)})_*$. This ends the proof of Lemma \ref{lemma-2}.
\end{proof}\medskip

\begin{proof}[Proof of Theorem B]
The theme of the proof is similar to that for Theorem A. We first consider the case that $\mathrm{ad}(y_0)$
has a positive eigenvalue. We choose the basis $\{ e_1,e_2\}$ of $\mathfrak{g}$ such that
$[e_1,e_2]=e_2$ and $y_0\in\mathbb{R}_{>0}$, and apply all above notations and calculations
in polar coordinate.

Plug (\ref{011}) and (\ref{016}) into (\ref{012}), we get the ODE charaterizing a left invariant conic Landsberg $(G,F)$,
\begin{eqnarray}\label{017}
\tfrac{{\rm d}}{{\rm d} t}\left(\tfrac{\tfrac{1}{f(t)}\tfrac{{\rm d}}{{\rm d} t}f(t)+\tfrac1{4f(t)}\tfrac{{\rm d}^3}{{\rm d}t^3}f(t)}{
\tfrac{1}{8f(t)^3}\left(2f(t)
\tfrac{{\rm d}^2}{{\rm d}t^2}f(t)-\left(\tfrac{{\rm d}}{{\rm d} t}f(t)\right)^2+4f(t)^2\right)^{3/2}}\right)=0.
\end{eqnarray}
For any initial value $(f(0),\tfrac{{\rm d}}{{\rm d} t}f(0),\tfrac{{\rm d}^2}{{\rm d}t^2}f(0),\tfrac{{\rm d}^3}{{\rm d}t^3}f(0))\in\mathcal{M}$, i.e., it satisfies
\begin{equation}
f(0)>0,\quad \tfrac{{\rm d}}{{\rm d}t}f(0)\neq0,\quad 2f(0)\tfrac{{\rm d}^2}{{\rm d}t^2}f(0)-\left(\tfrac{{\rm d}}{{\rm d} t}f(0)\right)^2+4f(0)^2>0,
\end{equation}
a unique $f(t)$ can be solved from (\ref{017}) for $t$ sufficiently close to $0$, such that
the corresponding $F=r\sqrt{2f}$ is a left conic Landsberg metric on $G$.

Nextly, we consider the case that $\mathrm{ad}(y_0)$ has a negative eigenvalue. The argument is similar. The equation characterizing the Landsberg property and the requirement for initial value
are unchanged.
%
%
%
\end{proof}

\subsection{Conic Berwald metrics}
\label{subsection-3-4}
In this subsection, we discuss when the left invariant conic Finsler metric $F=r\sqrt{2f(t)}$ in Section
 \ref{subsection-3-1} is Berwald.

\begin{lemma}\label{lemma-8}
Suppose that the left invariant $F=r\sqrt{2f(t)}$ is a conic Berwald metric defined for $t$ close to $0$,
which has nonvanishing spray vector field everywhere, then there exists $(a,b,c,d)\in\mathbb{R}^4\backslash\{0\}$ satisfying $a+d\geq0$, $ad-bc>0$ and $ac\neq0$, such that
$(ay^1+by^2)\tfrac{\partial}{\partial y^1}F+(cy^1+dy^2)\tfrac{\partial}{\partial y^2}F=0$ is valid wherever it is defined.
\end{lemma}

\begin{proof}
By Lemma \ref{lemma-7}, the connection operator $N$ of $F$ is a bilinear map. Since $\eta(e_1)=N(e_1,e_1)\neq0$, $N(\cdot,e_1)$ is a nonzero linear map. Denote $N(y,e_1)=(ay^1+by^2)e_1+(cy^1+dy^2)e_2$, with $(a,b,c,d)\in\mathbb{R}^4\backslash\{0\}$.
By (\ref{009}),
\begin{eqnarray*}
g_y(N(y,e_1),y)&=&(ay^1+by^2)g_y(e_1,y)+(cy^1+dy^2)g_y(e_2,y)\\
&=&
(ay^1+by^2)F(y)\tfrac{\partial}{\partial y^1}F(y)+(cy^1+dy^2) F(y)\tfrac{\partial}{\partial y^2}F(y)\\
&=&0,
\end{eqnarray*}
so we get $(ay^1+by^2)\tfrac{\partial}{\partial y^1}F+(cy^1+dy^2)\tfrac{\partial}{\partial y^2}F=0$
at each possible $y=y^1e_1+y^2e_2$.

Because we can replace $(a,b,c,d)$ by $(-a,-b,-c,-d)$, $a+d\geq0$ can be achieved.
By the strong convexity of $F$, its indicatrix $F=1$ can not be tangent to the $y^1$-axis, i.e.,
$F=1$ can be interpreted as a positive smooth function $y^1=y^1(y^2)$ which is defined for $y^2$ close to $0$. Then
\begin{equation}\label{022}
\tfrac{{\rm d}}{{\rm d}y^2}y^1(y^2)=-(\tfrac{\partial}{\partial y^1}F)^{-1}(\tfrac{\partial}{\partial y^2}F)=-\tfrac{2f(t)\sin t+\tfrac{{\rm d}}{{\rm d}t}f(t)\cos t}{2f(t)\cos t-\tfrac{{\rm d}}{{\rm d}t}f(t)\sin t}=\tfrac{ay^1+by^2}{cy^1+dy^2}
\end{equation} can be evaluated at $y_2=0$, which provides $c\neq0$. From (\ref{022}), we also see
that $\eta(y(0))\neq0$, i.e., $\tfrac{{\rm d}}{{\rm d}t}f(0)\neq0$, implies $a\neq0$.

Finally, the strong convexity of $F$ implies $y_1=y_1(y_2)$ has negative second derivative at $y_2=0$. So we have
$$\tfrac{{\rm d}^2}{{\rm d}y^2}y^1(0)=\tfrac{(a\tfrac{{\rm d}}{{\rm d }y^2}y^1(0)+b)\cdot(cy^1+dy^2)-(ay^1+by^2)\cdot(c\tfrac{{\rm d}}{{\rm d}y^2}y^1(0)+d)}{(cy^1+dy^2)^2}|_{y^2=0}=\tfrac{bc-ad}{c^2y^1(0)}<0,$$
which proves $ad-bc>0$.
\end{proof}

The indicatrix $F=1$ can be parametrized as $y^1(\theta)e_1+y^2(\theta)e_2$, such that
\begin{equation}\label{023}
\tfrac{{\rm d}}{{\rm d}\theta}\left(
                                  \begin{array}{c}
                                    y^1(\theta) \\
                                    y^2(\theta)\\
                                  \end{array}
                                \right)=\left(
                                          \begin{array}{cc}
                                            a & b \\
                                            c & d \\
                                          \end{array}
                                        \right)\left(
                                                 \begin{array}{c}
                                                   y^1(\theta) \\
                                                   y^2(\theta) \\
                                                 \end{array}
                                               \right).
\end{equation}
The solution of (\ref{023}) is $\left(
                                  \begin{array}{c}
                                    y^1(\theta) \\
                                    y^2(\theta)\\
                                  \end{array}
                                \right)=\exp(\theta\left(
                                          \begin{array}{cc}
                                            a & b \\
                                            c & d \\
                                          \end{array}
                                        \right))\left(
                                                 \begin{array}{c}
                                                   c_1 \\
                                                   c_2\\
                                                 \end{array}
                                               \right),$
where $c_1$ and $c_2$ are free parameters.
All possibilities of $A$ which satisfies the requirements in Lemma \ref{lemma-8} are the following.

{\bf Case 1}. The eigenvalues of $A$ are $\lambda_1 \pm\lambda_2\sqrt{-1}$, with $\lambda_1\geq0$ and $\lambda_2>0$. Then we can find a basis $\{\overline {e}_1,\overline{e}_2\}$ of $\mathfrak{g}$ such that 
$e_1\in\mathbb{R}_{>0}\overline {e}_1+\mathbb{R}_{>0}\overline{e}_2$,
and
$$(\lambda_1\overline{y}^1-\lambda_2\overline{y}^2)\tfrac{\partial}{\partial\overline{y}^1}F
+(\lambda_2\overline{y}^1+\lambda_1\overline{y}^1)\tfrac{\partial}{\partial\overline{y}^2}F=0$$
is satisfied everywhere. Here we use
$(\overline{y}^1,\overline{y}^2)$ and $(\overline{r},\overline{t})$ to denote the coordinates in
\begin{equation}\label{024}
y=\overline{y}^1\overline{e}_1+\overline{y}^2\overline{e}_2
=\overline{r}\cos\overline{t} \overline{e}_1
+\overline{r}\sin\overline{t} \overline{e}_2.
\end{equation}
Then 
we have the new polar coordinate representation $F=\overline{r}\sqrt{2\overline{f}(\overline{t})}$ with
$\overline{f}(\overline{t})=\mu e^{-2\lambda \overline{t}+\mu}$, in which $\lambda=\tfrac{\lambda_1}{\lambda_2}$ and $\mu>0$. In this case, the indicatrix $F=1$ is a segment of spiral curve when $\lambda>0$ and a segment of circle when $\lambda=0$.

All discussions in Section \ref{subsection-3-1} before (\ref{026}) are valid for  the new coordinates.
Apply Lemma \ref{lemma-9} to the case $\overline{f}(\overline{t})=\tfrac{1}{\sqrt{2}}e^{-2\lambda\overline{t}-2\mu}$, we can calculate
the spray vector field $\eta$ of $F=\overline{r}\sqrt{2\overline{f}(\overline{t})}$ and get
\begin{equation}\label{031}
\eta=-\tfrac{(\epsilon_1\overline{y}^1+\epsilon_2\overline{y}^2)
+\lambda(-\epsilon_2\overline{y}^1+\epsilon_1\overline{y}^2)}{\lambda^2+1}\cdot
\left((\lambda\overline{y}^1-\overline{y}^2)\overline{e}_1
+(\overline{y}^1+\lambda\overline{y}^2)\overline{e}_2\right),
\end{equation}
in which $\epsilon_1$ and $\epsilon_2$ are the real numbers in
$[\overline{e}_1,\overline{e}_2]=\epsilon_1\overline{e}_1+\epsilon_2\overline{e}_2$ (same below).
We see from (\ref{031}) that $\eta$ is quadratic, so $F$ is Berwald by Lemma \ref{lemma-7}.
In particular, $F$ is Riemannian if and only if $\lambda=\lambda_1=0$.

{\bf Case 2}. $A$ has two distinct eigenvalues $\lambda_1$ and$\lambda_2$ satisfying $\lambda_1>\lambda_2>0$. We can find a basis $\{\overline {e}_1,\overline{e}_2\}$ of $\mathfrak{g}$, such that $e_1\in\mathbb{R}_{>0}\overline {e}_1+\mathbb{R}_{>0}\overline{e}_2$ and
$\lambda_1\overline{y}^1\tfrac{{\partial}}{\partial\overline{y}^1}F+
\lambda_2\overline{y}^2\tfrac{\partial}{\partial \overline{y}^2}F=0$ is satisfied everywhere.
Here we apply the coordinates $(\overline{y}^1,\overline{y}^2)$ and
$(\overline{r},\overline{t})$
in (\ref{024}). Then we have the new polar coordinate representation
\begin{equation*}
F=\overline{r}\sqrt{2\overline{f}(\overline{t})}\quad\mbox{with}
\quad\overline{f}(\overline{t})=\mu \cos^{2-\lambda}\overline{t}\sin^\lambda\overline{t},
\end{equation*}
 in which
the constants $\mu>0$ and $\lambda=\tfrac{2\lambda_1}{\lambda_1-\lambda_2}>2$. In this case,
the indicatrix $F=1$ is a segment on the graph $\overline{y}^1=c'\overline{y}^{\lambda_1/\lambda_2}$ where $c'$ is a positive constant.

We use Lemma \ref{lemma-9} to calculate the spray vector field $\eta$ of $F$. Noticing that
$$\tfrac{1}{\overline{f}(\overline{t})}\tfrac{{\rm d}}{{\rm d}\overline{t}}\overline{f}(\overline{t})
=(\lambda-2)\tan\overline{t}+\lambda\cot\overline{t}=(\lambda-2)\cdot\tfrac{\overline{y}^2}{\overline{y}^1}
+\lambda\cdot\tfrac{\overline{y}^1}{\overline{y}^2},$$ we get
\begin{eqnarray*}
\eta&=&-\tfrac{\left(\tfrac{\overline{y}^2}{\overline{y}^1}+\tfrac{\overline{y}^1}{\overline{y}^2}\right)
\left(2\epsilon_1(2-\lambda)\overline{y}^2+2\epsilon_2\lambda\overline{y}^1\right)}{
(\lambda^2-2\lambda)\left(\tfrac{\overline{y}^2}{\overline{y}^1}+\tfrac{\overline{y}^1}{\overline{y}^2}\right)^2}
\cdot\left(-\tfrac{\lambda}{2\overline{y}^2}((\overline{y}^1)^2+(\overline{y}^2)^2)\overline{e}_1
+\tfrac{2-\lambda}{2\overline{y}^1}((\overline{y}^1)^2+(\overline{y}^2)^2)\overline{e}_2\right)\\
&=&-\tfrac{\epsilon_1(2-\lambda)\overline{y}^2+\epsilon_2\lambda\overline{y}^1}{
\lambda^2-2\lambda}\cdot(-\lambda\overline{y}^1\overline{e}_1+
(2-\lambda)\overline{y}^2\overline{e}_2),
\end{eqnarray*}
which is quadratic. So $F$ is Berwald in this case.

{\bf Case 3}. $A$ is not semi simple. Let $\lambda>0$ be the eigenvalue of $A$, then we can find
a basis $\{\overline{e}_1,\overline{e}_2\}$ of $\mathfrak{g}$, such that $e_1\in\mathbb{R}_{>0}\overline{e}_1+
\mathbb{R}_{>0}\overline{e}_2$ and $\lambda\overline{y}^1\tfrac{{\partial}}{\partial\overline{y}^1}F+
(\overline{y}^1+\lambda\overline{y}^2)\tfrac{\partial}{\partial \overline{y}^2}F=0$ is satisfied everywhere. Here we apply the coordinates $(\overline{y}^1,\overline{y}^2)$ and $(\overline{r},\overline{t})$ in (\ref{024}).
Using the new polar coordinate $(\overline{r},\overline{t})$,
we have the representation $F=\overline{r}\sqrt{2\overline{f}(\overline{t})}$, where
$\overline{f}(\overline{t})=\mu\cos^2\overline{t} e^{-2\lambda\tan\overline{t}}$ and
$\mu>0$. The indicatrix $F=1$ in this case is a segment on the graph $\overline{y}^2=c_1\overline{y}^1\ln\overline{y}^1+c_2\overline{y}^1$, where $c_1>0$ and $c_2$ are some constants.

We apply Lemma \ref{lemma-9} to calculate the spray vector field $\eta$ of $F$.
Using the fact $$ \tfrac{1}{\overline{f}(\overline{t})}\tfrac{{\rm d}}{{\rm d}\overline{t}}\overline{f}(\overline{t})
=-2\tan\overline{t}-2\lambda\sec^2\overline{t}=-\tfrac{2\overline{y}^2}{\overline{y}^1}-
2\lambda\left(1+\left(\tfrac{\overline{y}^2}{\overline{y}^1}\right)^2\right),$$
we see that
\begin{eqnarray*}
\eta&=&-\tfrac{4((\epsilon_1-\lambda\epsilon_2)\overline{y}^1+\lambda\epsilon_1\overline{y}^2)\cdot
\left(1+\left(\tfrac{\overline{y}^2}{\overline{y}^1}\right)^2\right)}{
4\lambda^2\left(1+\left(\tfrac{\overline{y}^2}{\overline{y}^1}\right)^2\right)^2
}\cdot\left(\lambda\overline{y}^1\left(1+\left(\tfrac{\overline{y}^2}{\overline{y}^1}\right)^2\right)
\overline {e}_1+(\overline{y}^1+\lambda\overline{y}^2)\left(1+\left(\tfrac{\overline{y}^2}{\overline{y}^1}\right)^2
\right)\overline{e}_2\right)\\
&=&-\tfrac{((\epsilon_1-\lambda\epsilon_2)\overline{y}^1+\lambda\epsilon_1\overline{y}^2)}{
\lambda^2
}\cdot\left(\lambda\overline{y}^1
\overline{e}_1+(\overline{y}^1+\lambda\overline{y}^2)\overline{e}_2\right)
\end{eqnarray*}
is quadratic. So $F$ is Berwald in this case.

Above case-by-case discussion can be summarized as the following proposition.
\begin{proposition}\label{prop-1}
For each $(a,b,c,d)\in\mathbb{R}^4$ satisfying $ac\neq0$ and $ad-bc>0$, a left invariant conic Berwald metric $F$ can be found, which satisfies
$(ay^1+by^2)\tfrac{\partial}{\partial y^1}F+(cy^1+dy^2)\tfrac{\partial}{\partial y^2}F=0$
and has a nowhere vanishing spray vector field around $e_1$.
\end{proposition}
\subsection{There exists no conic unicorn}
\label{subsection-3-5}
In this subsection, we prove Theorem D, i.e.,
any left invariant conic Landsberg metric on a 2-dimensional non-Abelian Lie group $G$ must be Berwald.

Firstly, we assume that the left invariant conic Landsberg metric $F_1=r\sqrt{2f_1(\theta)}$ is defined for $t$ sufficiently close to $0$ and its spray vector field is non-vanishing everywhere.
Here the coordinates $(y^1,y^2)$ and $(r,\theta)$ are
with respect to a basis of $\mathfrak{g}$ satisfying $[e_1,e_2]=\pm e_2$.\medskip

\noindent{\bf Claim}: $F_1$ is Berwald.\medskip

\begin{proof}[Proof of the claim]
We first prove the claim when $[e_1,e_2]=e_2$. The theme of the proof is the following.
By Theorem B and its proof in Section \ref{subsection-3-3}, we see that $F_1$ is uniquely determined by $(a_0,a_1,a_2,a_3)=(f_1(0),\tfrac{{\rm d}}{{\rm d}t}f_1(0),\tfrac{{\rm d}^2}{{\rm d}t^2}f_1(0),\tfrac{{\rm d}^3}{{\rm d}t^3}f_1(0))$, which satisfies
$$a_0>0, \quad a_1\neq0, \quad 2a_0a_2-a_1^2+4a_0^2>0.$$
We will look for $(a,b,c,d)\in\mathbb{R}^4$ satisfying $ac\neq0$ and $ad-bc>0$, such that the conic
$F_2=r\sqrt{2f_2(t)}$ in
$(ay^1+by^2)\tfrac{\partial}{\partial y^1}F_2+(cy^1+dy^2)\tfrac{\partial}{\partial y^2}F_2=0$
share the same initial value as $F_1$, i.e., $(f_2(0),\tfrac{{\rm d}}{{\rm d}t}f_2(0),\tfrac{{\rm d}^2}{{\rm d}t^2}f_2(0),\tfrac{{\rm d}^3}{{\rm d}t^3}f_2(0))=(a_0,a_1,a_2,a_3)$. By Theorem C or Proposition \ref{prop-1}, we see $F_2$ is Berwald and has a nowhere vanishing spray vector field when $t$ is sufficiently close to $0$. To summarize, both $F_1$ and $F_2$ are Landsberg with the same initial value in $\mathcal{M}$. So we see $F_1=F_2$ by Theorem B, i.e., $F_1$ is Berwald.

Now we assume $(f_2(0),\tfrac{{\rm d}}{{\rm d}t}f_2(0),\tfrac{{\rm d}^2}{{\rm d}t^2}f_2(0),\tfrac{{\rm d}^3}{{\rm d}t^3}f_2(0))=(a_0,a_1,a_2,a_3)$ and look for  the suitable $(a,b,c,d)$. Since scalar changes for $f_1(t)$, $f_2(t)$ and $(a,b,c,d)$ do not affect our discussion, we may assume $a_0=f_1(0)=f_2(0)=\tfrac12$
and $c=1$. Denote $h(t)=\tfrac{1}{2f_2(t)}\tfrac{{\rm d}}{{\rm d}t}f_2(t)$, then we have
the expansion
$$h(t)=c_0 +c_1 t+c_2 t^2+o(t^2),\mbox{ with } c_0=a_1, c_1=a_2-2a_1^2, c_2=\tfrac12a_3-3a_1a_2+4a_1^3.$$
The requirement (\ref{021}) provides $c_0\neq0$ and $c_1+c_0^2+1
>0$. The equality (\ref{022})
provides
\begin{eqnarray*}
\label{027}
& &-c_0-(c_1+c_0^2+1)t-(c_2+2c_1c_0+c_0^3+c_0)t^2+o(t^2)=-\tfrac{\tan t+h(t)}{1-u(t)\tan t}
\nonumber\\ &=&\tfrac{a+b\tan t}{1+d\tan t}=a+(b-ad)t+(ad^2-bd)t^2+o(t^2).
\end{eqnarray*}
Comparing the coefficients in (\ref{027}),  the following equations can be found,
\begin{eqnarray*}
a=-c_0,\quad, ad-b=c_1+c_0^2+1,\quad bd-ad^2=c_2+2c_1c_0+c_0^3+c_0,
\end{eqnarray*}
which has the following unique solution,
\begin{eqnarray*}
&&a=-c_0,\quad d=-\tfrac{bd-ad^2}{ad-b}=-\tfrac{c_2+2c_1c_0+c_0^3+c_0}{c_1+c_0^2+1}\\ &&b=ad-(c_1+c_0^2+1)=\tfrac{c_2c_0+2c_1c^2_0+c_0^4+c^2_0}{c_1+c_0^2+1}-c_1-c_0^2-1.
\end{eqnarray*}
Because $c_0\neq0$ and $c_1+c_0^2+1>0$, the requirements $a\neq0$ and $ad-bc=ad-b>0$ are satisfied.

To summarize, we have found the suitable $(a,b,c,d)$ such that the corresponding Berwald $F_2$ share the same initial value as $F_1$, so $F_1=F_2$ is Berwald. Above argument proves our claim when $[e_1,e_2]=e_2$. The prove when $[e_1,e_2]=-e_2$ is almost the same.
\end{proof}\medskip

Nextly, we consider a more general left invariant Landsberg metric $F$ on $G$. Let $y$ be any point in $$\mathcal{U}=\{y| y\in\mathfrak{g}\backslash[\mathfrak{g},\mathfrak{g}],
F \mbox{ is defined around }y \mbox{ with } \eta(y)\neq0\}.$$
By the claim we have just proved, $F$ is Berwald around $y$, i.e., the Berwald tensor $\mathbf{B}^i_{jkl}:\mathfrak{g}\backslash\{0\}\rightarrow\mathbb{R}$, $\forall i,j,k,l$, vanishes.

To extend the vanishing of $\mathbf{B}^i_{jkl}$ to $[\mathfrak{g},\mathfrak{g}]\backslash\{0\}$ and those points where
the spray vector field $\eta$ vanishes, we only need to observe that $\eta$ can not vanish identically in
any open subset $\mathcal{V}$ of $\mathfrak{g}\backslash\{0\}$. Assume conversely it happens. Since $\eta$ is positively 2-homogeneous, we may assume $\mathcal{V}$ is conic. Then for any $y\in\mathcal{V}$ satisfying $F(y)=1$, we have
$$g_y(y,[\mathfrak{g},\mathfrak{g}])=g_y(y,[\mathfrak{g},y])=g_y(\eta(y),\mathfrak{g})=0,$$
i.e., the tangent line of the indicatrix $F=1$ at $y$ is parallel to $[\mathfrak{g},\mathfrak{g}]$. This results in that
the intersection between $\mathcal{V}$ and the indicatrix of $F$ is a straight line segment, which contradicts the strong convexity of $F$.

To summarize, we can get by continuity the vanishing of the Berwald tensor $\mathbf{B}^i_{jkl}$ at $e\in G$. Finally, we can use the left invariancy of $F$ to get the vanishing of the Berwald tensor everywhere.
So the left invariant conic Landsberg metric $F$ is Berwald, which ends the proof of Theorem D.

\section{Homogeneous conic Landsberg Problem in the 2-dimensional case}

\subsection{Homogeneous conic Finsler manifold}

Let $G/H$ be a homogeneous manifold (here $G$ is a general Lie group). We denote $\mathfrak{g}=\mathrm{Lie}(G)$ and $\mathfrak{h}=\mathrm{Lie}(H)$. Suppose that $F$ is a $G$-invariant conic Finsler metric on $G/H$. The tangent space $T_o(G/H)$ at $o=eH$ can be
identified as $\mathfrak{g}/\mathfrak{h}$, such that the isotropic $H$-action on $T_o(G/H)$
is identified with the $\mathrm{Ad}_{\mathfrak{g}/\mathfrak{h}}(H)$-action on $\mathfrak{g}/\mathfrak{h}$.
Then $F$ is one-to-one determined by its restriction
$F=F(o,\cdot)$ at $o=eH$, which is an arbitrary $\mathrm{Ad}_{\mathfrak{g}/\mathfrak{h}}(H)$-invariant conic Minkowski norm
on $\mathfrak{g}/\mathfrak{h}$.

Without loss of generality, we may assume $G$ acts effectively on $G/H$, because otherwise we can replace $G$ and $H$ by their images in $\mathrm{Diff}(G/H)$. This effectiveness implies that $\mathfrak{h}$ does not contain any nonzero ideal of $\mathfrak{g}$. Notice that
the compactly imbedded property of $H$ or a reductive decomposition for the conic $(G/H,F)$ may not be guaranteed from the effectiveness assumption. It implies that, on one hand,
the homogeneous conic Finsler manifold $G/H$ may look very different from those globally defined ones, and on the other hand, when there is no reductive decomposition, the spray vector field and the curvatures of $(G/H,F)$ may be
hard to compute \cite{Xu2022-3}.

The following basic technique is still valid.
Suppose $L$ is a Lie subgroup of $G$, with $\mathfrak{l}=\mathrm{Lie}(L)$, such that
$\mathfrak{g}=\mathfrak{l}+\mathfrak{h}$, then the $G$-invariant conic Finsler metric $F$
naturally induces a $L$-invariant conic Finsler metric $F'$ on $L/(L\cap H)$ through the
linear isomorphism $\mathfrak{g}/\mathfrak{h}\cong\mathfrak{l}/\mathfrak{l}\cap\mathfrak{h}$.
The $L$-actions on $G/H$ realize a local isometry between $(G/H,F)$ and $(L/(L\cap H),F')$. In particular, if we further have $\mathfrak{l}\cap\mathfrak{h}=0$, $F'$ on $L/(L\cap H)$ can be lifted to a left invariant
conic Finsler metric $F''$ on $L$, such that the covering map from $L$ to $L/(L\cap H)$ provides a local isometry between $(L/(L\cap H),F')$ and $(L,F'')$. We summarize this technique as the following lemma.

\begin{lemma}\label{lemma-10}
Let $(G/H,F)$ be a homogeneous conic Finsler manifold. Suppose $L$ is a Lie group of $G$ such that
$\mathfrak{g}=\mathfrak{h}+\mathfrak{l}$, i.e., $\dim\mathfrak{l}-\dim(\mathfrak{l}\cap\mathfrak{h})=\dim G/H$, then $F$ is locally isometric to a $L$-invariant conic Finsler metric on $L/(L\cap H)$. Moreover, if $\mathfrak{l}\cap\mathfrak{h}=0$, then $F$ is locally isometric to a left invariant conic Finsler metric on $L$.
\end{lemma}

\subsection{Homogeneous conic Landsberg surface}
\label{subsection-4-2}
In this subsection, we prove Theorem F, i.e., each homogeneous conic Landsberg surface is Berwald.

We will need the following two lemmas for preparation.
\begin{lemma}\label{lemma-11}
Let $(G/H,F)$ be a homogeneous conic Finsler surface, i.e., $\dim G=\dim H+2$, in which $G$ is  connected. Suppose that the $\mathrm{ad}_{\mathfrak{g}/\mathfrak{h}}(\mathfrak{h})$-actions
on $\mathfrak{g}/\mathfrak{h}$ are all nilpotent. Then $\mathfrak{h}$ is an ideal of $\mathfrak{g}$, and $F$ is locally isometric to a left invariant conic Finsler metric on some Lie group.
\end{lemma}

\begin{proof} Since all $\mathrm{ad}_{\mathfrak{g}/\mathfrak{h}}(\mathfrak{h})$-actions on $\mathfrak{g}/\mathfrak{h}$ are all nilpotent,
we can use the Engel Theorem to find a basis $\{e_1,e_2\}$ of $\mathfrak{g}/\mathfrak{h}=T_o(G/H)$, such that
$\mathrm{ad}_{\mathfrak{g}/\mathfrak{h}}(\mathfrak{h})e_1=0$ $\mathrm{ad}_{\mathfrak{g}/\mathfrak{h}}(\mathfrak{h})e_2\subset\mathbb{R}e_1$.
There are two possibilities.

{\bf Case 1}. $\mathrm{ad}_{\mathfrak{g}/\mathfrak{h}}(\mathfrak{h})e_2=0$. Then $[\mathfrak{h},\mathfrak{g}]\subset\mathfrak{h}$, i.e., $\mathfrak{h}$  is an ideal of $\mathfrak{g}$. Let $H_0$ be the identity component of $H$, then $H_0$ is a closed normal subgroup of $G$. The covering map from $G/H_0$ to $G/H$ induces a $G$-invariant conic Finsler
metric $F'$ on $G/H_0$, such that $(G/H_0,F')$ and $(G/H,F)$ are locally isometric. Since $G/H_0$
is a Lie group itself, the $G$-invariant $F'$ is a left invariant conic Finsler metric on $G/H_0$.

{\bf Case 2}. $\mathrm{ad}_{\mathfrak{g}/\mathfrak{h}}(\mathfrak{h})e_2\neq0$. Let $H_0$ be the identity component of $H$. Then we have the following two orbits for the isotropic $H_0$-action on $T_o(G/H)=\mathfrak{g}/\mathfrak{h}$, $\mathcal{O}_\pm=\mathrm{Ad}_{\mathfrak{g}/\mathfrak{h}}(H_0) (\pm e_2)=(\pm e_2)+\mathbb{R}e_1$, which are two parallel straight lines. Let $\mathcal{A}_o\subset\mathfrak{g}/\mathfrak{h}$ be the non-empty conic open subset of $\mathfrak{g}/\mathfrak{h}\backslash\{0\}$, in which $F$ is defined. Then $\mathcal{A}_o\cap(\mathcal{O}_+\cup\mathcal{O}_-)\neq\emptyset$. Since $F$ is $\mathrm{Ad}_{\mathfrak{g}/\mathfrak{h}}(H)$-invariant, we see the indicatrix $F=1$ contains
at least a straight line. This contradicts the strong convexity of $F$.

To summarize, only Case 1 can happen. This ends the proof of Lemma \ref{lemma-10}.
\end{proof}

\begin{lemma}\label{lemma-13}
Let $G$ be a connected Lie group with $\mathfrak{g}=\mathrm{Lie}(G)$. Then for $u\in\mathfrak{g}$, the sum $\mathfrak{l}$ of linear subspaces $[\mathrm{Ad}(g)u,\mathfrak{g}]$ for all
$g\in G$ is an ideal of $G$.
\end{lemma}

\begin{proof} Each vector in $\mathfrak{l}$ is a finite linear combination $v=[w_1,u_1]+\cdots+[w_k,u_k]$, in which each $w_i\in\mathfrak{g}$ and each $u_i=\mathrm{Ad}(g_i)u$ for some $g_i\in G$. So to prove $\mathfrak{l}$ is an ideal, i.e., check $[\mathfrak{g},v]\subset\mathfrak{l}$, we only need to verify $[\mathfrak{g},[\mathfrak{g},u_i]]\subset\mathfrak{l}$ for each $i$. Proving $[\mathfrak{g},[\mathfrak{g},u]]\subset\mathfrak{l}$ is enough.

For each $g\in G$, $[\mathrm{Ad}(g)u,\mathfrak{g}]$ is the tangent space of the orbit $\mathrm{Ad}(G)u$ at $\mathrm{Ad}(g)u$. Since each tangent space of $\mathrm{Ad}(G)u$ is contained in $\mathfrak{l}$, and $G$ is connected, we have $\mathrm{Ad}(G)u\subset u+\mathfrak{l}$. For any $w\in\mathfrak{g}$,
$$\mathrm{Ad}(\exp tw) u=\exp \mathrm{ad}(tw)u=u+t[w,u]+\tfrac{t^2}{2}[w,[w,u]]+o(t^2)\in u+\mathfrak{l},$$
so $[w,[w,u]]\subset\mathfrak{l}$ for each $w\in\mathfrak{g}$.

For any $w_1$ and $w_2$ in $\mathfrak{g}$, we have seen that $[w_1,[w_1,u]]$, $[w_2,[w_2,u]]$, $[w_1+w_2,[w_1+w_2,u]]$ and $[[w_1,w_2],u]\in[\mathfrak{g},u]$ all belong to $\mathfrak{l}$. So
\begin{eqnarray*}
[w_1,[w_2,u]]=\tfrac{1}{2}([w_1+w_2,[w_1+w_2,u]]+[[w_1,w_2],u]-[w_1,[w_1,u]]-[w_2,[w_2,u]])
\end{eqnarray*}
also belongs to $\mathfrak{l}$, which ends the proof of Lemma \ref{lemma-13}.
\end{proof}

Let $(G/H,F)$ be a homogeneous conic Finsler surface, i.e., $\dim G=\dim H+2$.
Without loss of generality, we may asssume $G$ is connected and its action on $G/H$ is effective.
The following lemma can help us simplify the manifold.


\begin{lemma} \label{lemma-12}
$F$ is locally isometric to a Riemmanian symmetric metric on a 2-sphere or a left invariant conic Finsler metric on some Lie group.
\end{lemma}

\begin{proof} All possibilities for the Lie algebra $\mathfrak{g}=\mathrm{Lie}(G)$ are the following.

{\bf Case 1}. The radical $\mathfrak{r}$, i.e., the maximal solvable subalgebra of $\mathfrak{g}$ is nonzero. Denote by $\mathfrak{r}^{(0)}=\mathfrak{r}$, $\mathfrak{r}^{(i)}=
[\mathfrak{r}^{(i-1)},\mathfrak{r}^{(i-1)}]$, $\forall i>0$, the derived sequence of $\mathfrak{t}$. Each $\mathfrak{r}^{(i)}$ is an ideal of $\mathfrak{g}$. There exists a
non-negative integer $k$, such that $\mathfrak{r}^{(k)}\neq0$ is an Abelian ideal of $\mathfrak{g}$.
Obviously we have $\dim(\mathfrak{r}^{(k)}\cap\mathfrak{h})\geq\dim\mathfrak{r}^{(k)}-2$.
By the effectiveness assumption, $\mathfrak{r}^{(k)}$ is not contained in $\mathfrak{h}$, so
there are two subcases.

{\bf Subcase 1.1}. $\dim (\mathfrak{r}^{(k)}\cap\mathfrak{h})=\dim\mathfrak{r}^{(k)}-1$. Then we can find a 2-dimensional linear subspace $\mathrm{V}\subset\mathfrak{g}$, such that $\mathrm{V}\cap\mathfrak{h}=0$ and
$\dim\mathrm{V}\cap\mathfrak{r}^{(k)}=1$. Then $\mathfrak{l}=\mathrm{V}+\mathfrak{r}^{(k)}$
is Lie subalgebra of $\mathfrak{g}$ with $\dim\mathfrak{l}-\dim(\mathfrak{l}\cap\mathfrak{h})=2$.
Let $L$ be the connected Lie subgroup of $G$ generated by $\mathfrak{l}$, then by Lemma \ref{lemma-10}, $F$ is locally isometric to a $L$-invariant conic Finsler metric on $L/(L\cap H)$.
The $\mathrm{ad}_{\mathfrak{l}/(\mathfrak{l}\cap\mathfrak{h})}(\mathfrak{l}\cap\mathfrak{h})
$-action on $\mathfrak{l}/(\mathfrak{l}\cap\mathfrak{h})$ is nilpotent. So by Lemma \ref{lemma-11}, $F$ is locally isometric to a left invariant conic Finsler metric on some Lie group.

{\bf Subcase 1.2}. $\dim (\mathfrak{r}^{(k)}\cap\mathfrak{h})=\dim\mathfrak{r}^{(k)}-2$.
Let $\mathfrak{l}$ be any linear complement of $\mathfrak{r}^{(k)}\cap\mathfrak{h}$ in
$\mathfrak{r}^{(k)}$. Then it is a Lie subalgebra of $\mathfrak{g}$, which generates a connected Lie subgroup $L$ of $G$. By Lemma \ref{lemma-10}, $F$ is locally isometric to a left invariant conic Finsler metric on $L$.

{\bf Case 2}. $\mathfrak{g}$ is semi simple. There are four subcases.

{\bf Subcase 2.1}. $\mathfrak{g}$ has a nonzero compact simple ideal $\mathfrak{g}_1$. By the effectiveness of the $G$-action,  $\dim(\mathfrak{g}_1\cap\mathfrak{h})<\dim\mathfrak{g}_1$.  Meanwhile, we have
$\dim(\mathfrak{g}_1\cap\mathfrak{h})\geq \dim\mathfrak{g}_1-2$. So we must have
$\dim(\mathfrak{g}_1\cap\mathfrak{h})=1$ and $\mathfrak{g}_1=su(2)$ in this case. By Lemma \ref{lemma-10}, $F$ is locally isometric to a Riemannian symmetric metric on a 2-sphere.

{\bf Subcase 2.2}. $\mathfrak{g}$ is noncompact and simple, and $\mathfrak{g}$ is not isomorphic to $\mathfrak{sl}(2,\mathbb{R})$. Then it has a nonzero compact simple Lie subalgebra $\mathfrak{k}$. Let $u$ be any nonzero vector in $\mathfrak{k}$. Then there exists $g\in G$, such that $\mathrm{Ad}(g)u$ is not contained in $\mathfrak{h}$, because otherwise, by Lemma \ref{013},
$\mathfrak{h}$ contains a nonzero ideal of $\mathfrak{g}$, which contradicts the effectiveness assumption. So we can assume that the compact simple Lie subalgebra $\mathfrak{l}=\mathrm{Ad}(g)(\mathfrak{k})$ for some $g\in G$
is not contained in $\mathfrak{h}$.
For similar reasons as in Subcase 2.1, we can apply Lemma \ref{lemma-10} to the compact simple Lie group generated by
$\mathfrak{l}=\mathrm{Ad}(g)\mathfrak{k}$ and show that $F$ is locally isometric to a Riemannian symmetric metric on a 2-sphere. This is enough, though we believe we could get a contradiction in this subcase.

{\bf Subcase 2.3}. $\mathfrak{g}=sl(2,\mathbb{R})$. In this case $\dim\mathfrak{h}=1$. Let $u$ be any nonzero element in $\mathfrak{h}$. It can not be a nilpotent matrix in $sl(2,\mathbb{R})$, because otherwise the $\mathrm{ad}_{\mathfrak{g}/\mathfrak{h}}(\mathfrak{h})$-action on $\mathfrak{g}/\mathfrak{h}$ is nilpotent, and then $\mathfrak{h}$ is an ideal of $\mathfrak{g}$ by
Lemma \ref{lemma-11}, which is a contradiction. It can not be a semi simple matrix with real eigenvalues, because otherwise by the $\mathrm{Ad}(H)$-invariancy of $F$, the indicatrix $F=1$ in $\mathfrak{g}\backslash\mathfrak{h}$ contains a segment of hyperbolic curve centered at the origin, which
contradicts the strong convexity of $F$. To summarize, $u$ must be a semi simple matrix with purely imaginary eigenvalues. So $\mathfrak{h}$ has a zero intersection with the subalgebra $\mathfrak{l}$
of all upper triangular matrices in $sl(2,\mathbb{R})$. Lemma \ref{lemma-10} tells us that $F$ is
locally isometric to a left invariant conic Finsler metric on the Lie subgroup generated by $\mathfrak{l}$.

{\bf Subcase 2.4} $\mathfrak{g}$ is a semi simple Lie algebra of noncompact type, and it is not simple.
Let $\mathfrak{g}_1$ be a simple ideal of $\mathfrak{g}$. By the effectiveness assumption, $\dim\mathfrak{g}_1\cap\mathfrak{h}<\dim\mathfrak{g}_1$. If $\mathfrak{g}_1\cap\mathfrak{h}\leq\mathfrak{g}_1-2$, then the equality must happen, and by similar argument as in Subcase 2.2 or Subcase 2.3, we can see that $F$ is locally isometric to a Riemannian symmetric metric on a 2-sphere or a left invariant conic Finsler metric on a Lie group.
Now suppose $\dim(\mathfrak{g}_1\cap\mathfrak{h})=\dim\mathfrak{g}_1-1$. We can find
a 2-dimensional linear subspace $\mathrm{V}$ of $\mathfrak{g}$, such that $\mathrm{V}\cap\mathfrak{h}=0$ and $\dim\mathrm{V}\cap\mathfrak{g}_1=1$. Then $\mathfrak{g}'=\mathfrak{g}_1+\mathrm{V}$ is a Lie subalgebra of $\mathfrak{g}$, and
$\dim(\mathfrak{g}'\cap\mathfrak{h})=\dim\mathfrak{g}-2$, i.e., $\mathfrak{g}'\cap\mathfrak{h}\subset\mathfrak{g}_1$. There exists $v\in\mathrm{V}\backslash\{0\}$, such that $\mathfrak{g}'=\mathfrak{g}_1\oplus\mathbb{R}v$ is
a Lie algebra decomposition. Take any $u\in\mathfrak{g}_1\backslash\mathfrak{h}$,
then $\mathfrak{l}=\mathrm{span}\{u,v\}$ is a 2-dimensional Abelian Lie subalgebra such that $\mathfrak{l}\cap\mathfrak{h}=0$. By Lemma \ref{lemma-10}, $F$ is locally isometric to
a left invariant conic Finsler metric on the Lie subgroup generated by $\mathfrak{l}$.

To summarize, above case-by-case discussion has exhausted all possibilities for the Lie group $G$, and proved
Lemma \ref{lemma-12} in each case.
\end{proof}
\medskip

\begin{proof}[Proof of Theorem F]
Let $(M,F)$ be a homogeneous conic Landsberg surface. By Lemma \ref{lemma-12}, it is locally isometric to one of the following:
 \begin{enumerate}
 \item a Riemannian symmetric metric on a 2-sphere;
 \item a left invariant conic Finsler metric on a 2-dimensional Abelian Lie group;
 \item a left invariant conic Finsler metric on a 2-dimensional non-Abelian Lie group.
 \end{enumerate}
In the first two cases, the metric is either Riemannian or local (conic) Minkowski. Tt is Berwald in these two cases. In the third case, we use the Landsberg assumption and Theorem D to get $F$ is Berwald. To summarize, we have verified Berwald property of $F$ for each possibility. This ends the proof of Theorem E.
\end{proof}

\bigskip

\noindent{\bf Acknowledgement}\quad
This paper is supported by Beijing Natural Science Foundation (1222003) and
National Natural Science Foundation of China (12131012, 11821101). The author sincerely thank Akbar Tayebi for the helpful discussion which inspired this work.

\end{document}